\documentclass{article}[12pt]
\usepackage{amsmath,amsfonts,amsthm}
\usepackage{amssymb}

\usepackage{epsfig}

\usepackage{subfigure}
\usepackage{graphics}
\usepackage{graphicx}
\usepackage{color}
\usepackage{soul}
\usepackage[makeroom]{cancel}

\usepackage{latexsym}
\usepackage{graphics}
\usepackage{amsmath}
\usepackage{amsthm}
\usepackage{xspace}
\usepackage{amssymb}
\usepackage{color}
\usepackage{cite}
\usepackage{latexsym}
\usepackage{graphics}
\usepackage{amsmath}
\usepackage{amsthm}
\usepackage{xspace}
\usepackage{amssymb}
\usepackage{epsfig}
\usepackage{comment}

\newcommand{\beql}[1]{\begin{equation}\label{#1}}
\newcommand{\eeql}{\end{equation}}
\newcommand{\eqn}[1]{(\ref{#1})}

\newcommand{\R}{\mathbb{R}}
\newcommand{\pr}{\mathbb{P}}
\newcommand{\E}{\mathbb{E}}

\newcommand{\ci}{{\cal I}}

\newcommand{\ck}{{\cal K}}
\newcommand{\cx}{{\cal X}}
\newcommand{\cm}{{\cal M}}

\newcommand{\bk}{\boldsymbol{k}}
\newcommand{\bx}{\boldsymbol{x}}
\newcommand{\bX}{\boldsymbol{X}}

\newcommand{\be}{\boldsymbol{e}}

\newcommand{\bw}{\boldsymbol{w}}
\newcommand{\bZero}{\boldsymbol{0}}

\newcommand{\veps}{\varepsilon}

\newtheorem{thm}{Theorem}
\newtheorem{lem}[thm]{Lemma}
\newtheorem{prop}[thm]{Proposition}

\newtheorem{definition}[thm]{Definition}

\theoremstyle{remark}
\newtheorem{remark}{Remark}

\begin{document}

\title{An infinite server system with packing constraints and ranked servers
}

\author
{
Alexander L. Stolyar \\
University of Illinois at Urbana-Champaign\\
Urbana, IL 61801, USA \\
\texttt{stolyar@illinois.edu}
}

\date{\today}

\maketitle

\begin{abstract}

A service system with multiple  types of customers, arriving as Poisson processes, is considered. 
The system has infinite number of servers, ranked by $1,2,3, \ldots$; a server rank is its ``location."
Each customer has an independent exponentially distributed service time, 
with the mean determined by its type. 
Multiple customers (possibly of different types) can be placed for service into 
one server, subject to ``packing'' constraints.
Service times of different customers are independent, even if served simultaneously
by the same server.   
The large-scale asymptotic regime is considered,
such that the mean number of customers $r$ goes to infinity.

We seek algorithms with the underlying objective of minimizing the location (rank) $U$ of the right-most (highest ranked) occupied (non-empty) server. Therefore, this objective seeks to minimize the total number $Q$ of occupied servers {\em and} keep the set of occupied servers as far at the ``left'' as possible, i.e., keep $U$ close to $Q$. In previous work, versions of 
{\em Greedy Random} (GRAND) algorithm have been shown to asymptotically minimize $Q/r$ as $r\to\infty$.
In this paper we show that when these algorithms are combined with the First-Fit rule for ``taking'' empty servers,
they asymptotically minimize $U/r$ as well.

\end{abstract}


%Mathematics 

%\newpage

\section{Introduction}
\label{sec-intro}

\subsection{Motivation and main results (informally)}

We consider the following service system, whose primary motivation
is the problem of efficient real-time placement 
of virtual machines into physical machines in a cloud data center.
(A general discussion of this problem in data centers can be found in, e.g., \cite{Gulati2012}.)
There are multiple customer (virtual machine) types $i$, each arriving as Poisson process of rate $\Lambda_i$.
Customers are served by a countable number of servers (physical machines), ranked by $1,2,3, \ldots$; the ranks can be viewed as server ``locations.'' Each customer has to be placed for service immediately upon arrival, and has an independent exponential service time, with the mean $1/\mu_i$ determined by its type $i$. Multiple customers can be served simultaneously by the same server, subject to certain general {\em packing constraints.}  
A customer placement algorithm decides which server a new arriving customer is placed into, among those servers that may still accept
it without violating packing constraints. 
Customer ``migration'' -- i.e., moving from one server to another during service -- is not allowed.
The basic objective is to minimize in steady-state the location (rank) $U$ of the right-most (highest-ranked) occupied (i.e., non-empty) server. 
Therefore, we seek to minimize the total number $Q$ of occupied servers {\em and} keep the set of occupied servers as far at the ``left'' as possible, i.e., keep $U$ close to $Q$. 

The large-scale asymptotic regime is considered,
such that $\Lambda_i = \lambda_i r$, where $\lambda_i>0$ are constants and the scaling parameter $r$ goes to infinity. 
Without loss of generality it can be assumed that $\sum_i \lambda_i /\mu_i =1$, so that parameter $r$ is the mean number of customers in the system 
in steady state.

The same model, but without servers having ranks, was considered in \cite{StZh2013,StZh2015,St2015_grand-het}
(see also references therein). 
In particular, papers  \cite{StZh2013,StZh2015} introduce an analyze a very simple algorithm, called Greedy-Random (GRAND),
which does not have notion of server ranks/locations, and only tries to minimize the number of occupied servers $Q$.
This algorithm places an arriving customer uniformly at random into a set of servers, consisting of all currently occupied servers that can still 
fit it and a certain number $X_{\bZero}$ of empty servers, where $X_{\bZero}$ depends on the current system state; GRAND($X_{\bZero}$) refers to the algorithm instance with specific function $X_{\bZero}$. Papers \cite{StZh2013} and \cite{StZh2015} study algorithm instances GRAND($aZ$) and GRAND($Z^p$),
respectively, where $Z$ is the current total number of customers in the system, and $a>0$ and $0<p<1$ are fixed parameters.
A key feature of GRAND is that it is very parsimonious: it does {\em not} need to know customer types' arrival rates $\Lambda_i$ or service rates $\mu_i$; it does {\em not} need to a priori solve any optimization problem; at a time of a customer arrival, it only needs to know which servers can fit it  (and nothing else about the structure of packing constraints). 

It is proved in \cite{StZh2013,StZh2015} that GRAND($aZ$) and GRAND($Z^p$), 
are in fact asymptotically optimal in the sense of convergence in distribution 
\beql{eq-q-opt}
Q/r \Rightarrow q^* ~\mbox{as}~ r\to\infty,
\eeql
where constant $q^*$ is the lower bound under any algorithm, even allowing migration. 
(To be precise, $Q/r \Rightarrow q^*$ under GRAND($Z^p$); under GRAND($aZ$), $Q/r \Rightarrow q^{*,a}$, where $q^{*,a} \to q^*$
as $a\downarrow 0$.) Informally speaking, when $r$ is large, under these instances of GRAND the system dynamics is such that 
the state converges to the (vicinity of) the optimal state; and this occurs despite the parsimonious nature of the algorithm.

Thus, in terms of a data center application, GRAND($aZ$) and GRAND($Z^p$) allow one to
asymptotically minimize the total number of physical machines in use. However, those results do not answer the question of whether or not 
with high probability
all  occupied physical machines can be ``kept'' within one physical data center, consisting of $N$ physical machines, with $N > q^* r$,
say $N=(q^* + \delta)r$. The answer to this latter question is positive if, under some algorithm, not only 
\eqn{eq-q-opt} holds, 
but also the stronger property
\beql{eq-u-opt}
U/r \Rightarrow q^* ~\mbox{as}~ r\to\infty.
\eeql
Obviously, establishing property \eqn{eq-u-opt} is more involved, because the set of occupied servers, located between $1$ and $U$, 
is ``fragmented'' in a complicated random way by empty servers. Analysis of the asymptotic behavior of $U$ is challenging even for the
M/M/$\infty$ queueing system with ranked servers, 
which is a special case of our model, such that there is only one customer type and each server can serve only one customer at a time.
(See Section~\ref{sec-rel-prev-work}.) Such analysis is even more challenging for the far more general model in this paper,
where server occupancy times (from an empty server being taken to becoming empty again) are no longer independent. 

In this paper we consider the placement algorithms which combine GRAND($aZ$) and GRAND($Z^p$) with the  
{\em First-Fit} (FF) rule for allocation 
of empty servers. Under FF rule, every time an algorithm decides to place a customer into an empty server, the left-most empty server is chosen.
We show that these combinations, 
labeled GRAND($aZ$)-FF and GRAND($Z^p$)-FF, are asymptotically optimal in that  \eqn{eq-u-opt} holds.

In fact, our results are more generic -- we separate the problem of keeping $U/r$ close to $Q/r$ from the problem of minimizing $Q/r$.
Our main contributions are as follows:
\begin{itemize}
\item[(i)] Suppose there is an algorithm ALG, oblivious of server ranks, under which the number of occupied servers $Q/r \Rightarrow \hat q$, 
for some constant $\hat q \ge q^*$. Let ALG-FF label the combination of ALG with the FF rule for taking empty servers.
We  give sufficient conditions (Theorem~\ref{th-reduction-uni}) for ALG, under which the ALG-FF will keep $U/r$ close to $Q/r$. 
The conditions are such that, roughly speaking, in steady-state, over a long enough time $\beta(r)$, with high probability: $Q/r$ stays close to $\hat q $; each occupied server empties at the rate at least $\alpha(r)$; and $r e^{-\alpha(r) \beta(r)} \to 0$ as $r \to \infty$. 

\item[(ii)] We verify sufficient conditions (i) for algorithms GRAND($aZ$) and GRAND($Z^p$),
 thus proving the asymptotic optimality of GRAND($aZ$)-FF and GRAND($Z^p$)-FF, respectively. 
Specifically, we show
\beql{eq-intro-zp}
|U-q^* r|/r \Rightarrow 0 ~~\mbox{under GRAND($Z^p$)-FF ~~(Theorem~\ref{th-grand-zp-ff})},
\eeql
and 
\beql{eq-intro-az}
|U-q^{*,a} r|/r^{1/2+\veps} \Rightarrow 0, ~\forall \veps>0, ~~\mbox{under GRAND($aZ$)-FF ~~ (Theorem~\ref{th-grand-fluid})},
\eeql
where (by the results of \cite{StZh2013}) $q^{*,a}$ depends on the algorithm parameter $a>0$ and is such that $Q/r \Rightarrow q^{*,a}$ as $r\to\infty$, and $q^{*,a} \to q^*$ as $a\downarrow 0$.
\end{itemize}

\begin{remark}
\label{rem-stronger-weaker}
Using the fact that the total number of occupied servers $Q$ is upper bounded by the total number of customers,
which is Poisson with mean $r$, it is easy to see that \eqn{eq-intro-zp} 
and \eqn{eq-intro-az} 
imply $\E Q/r \to q^*$  and $\E Q/r \to q^{*,a}$, respectively. Then, 
\eqn{eq-intro-zp} and \eqn{eq-intro-az} imply the following weaker asymptotic optimality properties:
\beql{eq-intro-zp-weaker}
\E [q^* r-G(q^* r)]/r \to 0 ~\mbox{and}~ \E [Q-G(q^* r)]/r \to 0 ~~~\mbox{under GRAND($Z^p$)-FF},
\eeql
\beql{eq-intro-az-weaker}
\E [q^{*,a} r-G(q^{*,a} r)]/r \to 0 ~\mbox{and}~  \E [Q-G(q^{*,a} r)]/r \to 0  ~~~\mbox{under GRAND($aZ$)-FF},
\eeql
where $G(N)$ is the number of occupied servers with ranks at most $N$. 
Property \eqn{eq-intro-zp-weaker}
[respectively, \eqn{eq-intro-az-weaker}] states that there are only $o(r)$ empty servers to the left of $q^* r$ [respectively, $q^{*,a} r$]
and only $o(r)$ occupied servers to the right of $q^* r$ [respectively, $q^{*,a} r$].
\end{remark}

\begin{remark}
\label{rem-gen2}
In paper \cite{St2015_grand-het} the GRAND($aZ$) results of \cite{StZh2013} are generalized to a heterogeneous system, where servers can be of multiple types, with the packing constraints depending on the server type. Our Theorem~\ref{th-grand-fluid} for GRAND($aZ$)-FF and its proof generalize to the model in \cite{St2015_grand-het} in a fairly straightforward fashion -- we do not do it in this paper to simplify the exposition. 
\end{remark}

\subsection{Previous work}
\label{sec-rel-prev-work}

We already discussed that this paper extends and complements the results of \cite{StZh2013,StZh2015} on the 
GRAND($aZ$) and GRAND($Z^p$) algorithms. Paper \cite{St2015_grand-het} generalizes the GRAND($aZ$) results of \cite{StZh2013} to a heterogeneous service system where servers can be of multiple types, with the packing constraints depending on the server type. (We already remarked that our Theorem~\ref{th-grand-fluid} on GRAND($aZ$)-FF generalizes to the model in \cite{St2015_grand-het}.) Paper \cite{St2015_grand-het} also considers a different variant of the model, where
there is a finite set of servers, and arriving customers may be blocked. 
(Thus, this is another way to model a data center with finite number of physical machines.)
For this variant of the model, \cite{St2015_grand-het} assumes strictly subcritical case, 
with number of servers $N=(1+\delta) q^* r$, $\delta>0$,
and considers the algorithm (which can be viewed as another version of GRAND), simply assigning an arriving customer to any server available to it, uniformly at random, and blocks the customer if none is available.
The paper proves local stability of fluid limits at the unique equilibrium point -- this strongly suggests (but does not prove) that the steady-state blocking probability vanishes as $r\to\infty$.

There has been a significant amount of work (see \cite{Kosten37,Newell-book,CKS86,Aldous, Preater97, K00, SK08} and references therein)
on the M/M/$\infty$ queueing system with ranked servers. 
This system is a special case of our model, such that there is only one customer type and the packing constraints are trivial -- each server can serve exactly one customer. In particular, here obviously $q^* = 1$.   
A placement algorithm only needs to pick an empty 
server for an arriving customer, 
and \cite{Kosten37,Newell-book,CKS86,Aldous, Preater97,K00, SK08} study the FF algorithm.
The exact distribution of $U$ was found in \cite{Kosten37,CKS86,Preater97}, but in the form of an infinite sum, which is not easy to analyze. 
Paper \cite{CKS86} proves asymptotic optimality of FF in the form
\beql{eq-coffman-asymp}
\mbox{$\E U- r \le c (r \log r)^{1/2}$, for some constant $c$, ~~for large $r$.}
\eeql
Papers \cite{Aldous,K00,SK08} provide a variety of asymptotic results, in particular they derive the asymptotics 
$U - r \sim (2 r \log \log r)^{1/2}$, 
which may be considered a refinement of \eqn{eq-coffman-asymp}, except it is in terms
of convergence in distribution (like \eqn{eq-intro-zp} and \eqn{eq-intro-az} are), not in the stronger sense of 
expectation bound as in \eqn{eq-coffman-asymp}.

There is also a line of work (see \cite{CL89,ErSt2024} and references therein) on the extension of the M/M/$\infty$ with ranked servers model
in the following direction: there are multiple customer types $i$ having different ``sizes'' $s_i$; a type-$i$ customer needs to be placed on a contiguous set of $s_i$ empty servers. The difficulty here is that the left-most empty server cannot always be taken, and this is a major additional source of ``fragmentation'' 
of the set of occupied servers. In this model FF algorithm places an arriving type-$i$ customer into the left-most size-$s_i$ contiguous set of empty servers; and there is {\em no} simple reason why FF would be optimal in the sense of minimizing $U$.
Paper \cite{CL89}, in particular, provides a universal lower bound on $U$ under any placement algorithm.
Recent paper \cite{ErSt2024}, 
for the case of two customer types, with sizes 1 and 2, proves the 
asymptotic optimality of FF in the (weaker) form \eqn{eq-intro-zp-weaker}.

\subsection{Basic notation used throughout the paper}
\label{subsec-notation}

Sets of real and real non-negative numbers are denoted by $\R$ and $\R_+$, respectively.
 We use bold and plain letters  for vectors and scalars, respectively.
The standard Euclidean norm of a vector $\bx\in \R^n$ is denoted by $\|\bx\|$. 
Convergence $\bx \to \bw \in \R^n$ means ordinary convergence in $\R^n$,
while $\bx \to W \subseteq \R^n$ means convergence to a set, namely,
$\inf_{\bw\in W} \|\bx-\bw\|\to 0$.
The $i$-th coordinate unit vector in $\R^n$ is denoted by $\be_i$.
We denote by $\nabla F(\bx)$ and  $\nabla^2 F(\bx)$ the gradient and the Hessian of a function $F(\bx), ~\bx \in \R^n$.

Symbol $\implies$
denotes convergence in distribution of random variables taking values in space $\R^n$
equipped with the Borel $\sigma$-algebra. The
abbreviation {\em w.p.1} means {\em with probability 1}.
We often write $x(\cdot)$ to mean the function (or random process) $\{x(t),~t\ge 0\}$.
Random element $x(\infty)$ denotes the value of the process $x(t)$, when it is in stationary regime;
in other words, the distribution of $x(\infty)$ is the stationary distribution of the process.

Notation $\lceil \zeta \rceil$ means the smallest integer
greater than or equal to $\zeta$, and $\lfloor \zeta \rfloor$ means the largest integer 
smaller than or equal to $\zeta$; $\zeta \wedge \eta = \min(\zeta,\eta)$, $\zeta \vee \eta = \max(\zeta,\eta)$.
We will use notation $C_g$ for a generic positive constant, its value may be different in different expressions.
For a set $\ci$, $|\ci|$ is its cardinality.
Abbreviation WLOG means {\em without loss of generality}.

\subsection{Layout of the rest of the paper}

The formal model is described in Section~\ref{sec-model}, while Section~\ref{sec-results-infinite} defines the algorithms and
presents our main results (Theorems~\ref{th-grand-fluid} and \ref{th-grand-zp-ff}), along with the necessary background and notation. 
Section~\ref{sec-reduction-uni} gives a generic result (Theorem~\ref{th-reduction-uni}) on the
combination of any algorithm, oblivious of server ranks, with First-Fit rule for choosing empty servers.
Sections~\ref{sec-fsp-dynamics} and \ref{sec-zp-ff-proof} contain the proofs of Theorems~\ref{th-grand-fluid} and \ref{th-grand-zp-ff},
respectively. Some conclusions and discussion are given in Section~\ref{sec-discussion}.

\section{Model}
\label{sec-model}

We consider a service system with $I$  types 
of customers, indexed by $i \in \{1,2,\ldots,I\} \equiv \ci$. 
The service time of a type-$i$ customer is an exponentially distributed random variable with mean $1/\mu_i$.
All  customers' service times are mutually independent.
There is an infinite, countable ``supply'' of servers.
A server can potentially serve more than one customer simultaneously, subject to the following very general packing constraints. We say that a vector $\bk = (k_1,\ldots,k_I)$ with non-negative integer
$k_i, ~i\in \ci,$ is a server {\em configuration}, if a server can simultaneously serve a combination of customers of different types given by the values $k_i$. 
There is a finite set of all allowed server configurations, denoted by $\bar\ck$.
We assume that $\bar\ck$ satisfies a natural {\em monotonicity} condition: if 
$\bk\in \bar\ck$, then all ``smaller'' configurations $\bk'= (k'_1,\ldots,k'_I)$, i.e. such that $k'_i \le k_i$ for all $i$, belong to $\bar\ck$ as well. Without loss of generality, assume that for each $i$, $\be_i \in \bar\ck$, where $\be_i$ is the $i$-th coordinate unit vector (otherwise, type-$i$ customers cannot be served at all).
By convention, vector $\bZero \in \bar\ck$, where
$\bk=\bZero$ is the $I$-dimensional component-wise zero vector -- this is the configuration of an empty 
server. We denote by $\ck=\bar\ck \setminus \{\bZero\}$ 
the set of server configurations {\em not} including the empty (or, zero) configuration. 

An important feature of the model is that simultaneous service does {\em not} affect the service time
distributions of individual customers. In other words, the service time of a customer is unaffected by whether or not there are other customers served simultaneously by the same server. A customer can be ``added'' to an empty or occupied server, as long as the packing constraints are not violated. Namely, a type $i$ customer can be added to a server whose current configuration $\bk\in\bar\ck$  is such that $\bk+\be_i \in \ck$. When the service of a type-$i$ customer by a  server in configuration $\bk$ is completed,
the customer leaves the system and the server's configuration changes to $\bk-\be_i$.

Customers of type $i$ arrive as an independent Poisson process of rate $\Lambda_i >0$;
these arrival processes are independent of each other and of the customer service times.
Each arriving customer is immediately placed for service in one of the servers, as long as packing 
constraints are not violated; after that the customer stays in that server until its service is completed -- there is {\em no customer ``migration''} during the service.

While the server capabilities are identical, they have unique ranks $\ell=1,2, \ldots$, which can be
viewed as server ``locations;'' accordingly, we will say that server $\ell$ is located to the left (resp., right) of server $m$, if $\ell \le m$ (resp. $\ell > m$).

\section{Main results}
\label{sec-results-infinite}

In this section we formally define the proposed placement algorithms, the asymptotic regime, and state 
our main 
results.

Denote by $\bar X_{\bk}(t;s)$ the number of servers in configuration $\bk\in \ck$, located to the left of $s\ge 0$ (i.e. with ranks $\ell \le s$) at time $t$. Clearly, $\bar X_{\bk}(t;s)$ is piece-wise constant, right-continuous, non-decreasing in $s$, with $\bar X_{\bk}(t;0)=0$, and 
$\bar X_{\bk}(t;\infty)\doteq \lim_{s\uparrow \infty} \bar X_{\bk}(t;s)$ being the total number of servers in configuration $\bk\in \ck$.
The system state at time $t$ is then the set of functions $\bar \bX(t) = \{\bar X_{\bk}(t;\cdot), ~\bk\in \ck\}$. 

A {\em placement algorithm} determines which server an arriving customer is placed to, as a function of the current system state $\bar \bX(t)$. Under any well-defined placement algorithm, 
the process $\{\bar \bX(t), t\ge 0\}$ is a continuous-time Markov chain 
with a countable state space. It is easily seen to be irreducible and positive recurrent. Indeed, 
 the total number $Y_i(t)$ of type-$i$ customers in the system is independent from the  placement algorithm and
is a Markov chain corresponding to an $M/M/\infty$ system; moreover, these Markov chains are independent across $i$;
therefore, the Markov chain $(Y_i(t), i\in \ci), t\ge 0,$ is positive recurrent; therefore,
empty state (with all $Y_i=0$) is reached from any other, and the expected time to return to it is finite.
(Note that the stationary distribution of $Y_i(\cdot)$ is 
Poisson with mean $\Lambda_i/\mu_i$; we denote by $Y_i(\infty)$ the random value of $Y_i(t)$ in steady-state -- it is, therefore, a Poisson random 
variable with mean $\Lambda_i/\mu_i$.)
Consequently, the process $\{\bar \bX(t), ~t\ge 0\}$ has a unique stationary distribution; let $\bar \bX(\infty)$
be the random system state $\bar \bX(t)$ in stationary regime.

Let us denote by $U(t)$ the location of the right-most occupied (non-empty) server at time $t$, that is 
$$
U(t) = \min\{s ~| \sum_{\bk\in \ck } \bar X_{\bk}(t;s) = \sum_{\bk\in \ck } \bar X_{\bk}(t;\infty) \}.
$$
We are interested in finding a placement algorithm that minimizes $U(\infty)$ in the stationary regime. Informally speaking, we seek an algorithm 
which, first, optimally ``packs'' customers into servers so as to minimize the total number of occupied servers and, in addition, keeps the occupied servers ``packed'' as much as possible ``on the left'' so as to minimize $U(t)$. 

We will use notation $X_{\bk}(t)\doteq \bar X_{\bk}(t;\infty)$ for the total number of servers in configuration $\bk\in \ck$ at time $t$,
and notation $\bX(t) = \{X_{\bk}(t), ~\bk\in \ck\}$ for the projection of $\bar \bX(t)$, containing only the information about the configurations of occupied servers (without information about their ranks/locations). 

We now define two algorithms, labeled GRAND($aZ$)-FF and GRAND($aZ$)-FF, which are, respectively, the GRAND($aZ$) \cite{StZh2013}
and GRAND($Z^p$) \cite{StZh2015} algorithms (for the model oblivious of server ranking), augmented for the model in this paper (with server ranking) by the First-Fit (FF) rule for ``taking'' empty servers.

\begin{definition}[GRAND($a Z$)-FF algorithm]
\label{df:grand}
(i) Rank-oblivious part 
(GRAND($a Z$)). The algorithm has a single parameter  $a>0$.
Let $Z(t)=\sum_i \sum_{\bk} k_i X_{\bk}(t)$ denote the total number of customers in the system at time $t$.
Denote $X_{\bZero}(t) \doteq \lceil a Z(t) \rceil$, 
$$
X_{(i),\diamond}(t) \doteq  \sum_{\bk\in \ck:~\bk+\be_i\in \ck} X_{\bk}(t),
$$
and $X_{(i)}(t) \doteq  X_{\bZero}(t) + X_{(i),\diamond}(t)$. ($X_{(i),\diamond}(t)$ is the number of occupied servers,
available to new type-$i$ customers.) 
If $X_{(i)}(t) \ge 1$, a new customer of type $i$, arriving at time $t$, is placed into an empty server with probability 
$X_{\bZero}(t)/X_{(i)}(t)$, and with probability 
$X_{(i),\diamond}(t)/X_{(i)}(t)$ it is placed uniformly at random into one of the $X_{(i),\diamond}(t)$ occupied servers available to type $i$.
If $X_{(i)}(t)=0$,  
an arriving customer is  
placed into an empty server.
\\ (ii) Empty server selection part (FF). If the algorithm in part (i) chooses to place a customer into an empty server, the lowest-ranked (left-most) empty server is taken.
\end{definition}

\begin{definition}[GRAND($Z^p$)-FF algorithm]
\label{df:grand-zp}
The algorithm has a single parameter  $0<p<1$.
This algorithm is defined exactly as GRAND($aZ$)-FF, except $X_{\bZero}(t) \doteq \lceil (Z(t))^p \rceil$.
(GRAND($Z^p$)-FF can be interpreted as GRAND($aZ$)-FF with $a$ being not a fixed parameter, but rather the function $a=Z^{p-1}$ of the system current state.)
\end{definition}

Clearly, under GRAND($a Z$)-FF [resp., GRAND($Z^p$)-FF] algorithm the process $\bX(t)$, which is a projection of $\bar \bX(t)$,
is itself a positive recurrent Markov chain, exactly same as the Markov chain under GRAND($a Z$) [resp., GRAND($Z^p$)],
studied in \cite{StZh2013} [resp., \cite{StZh2015}].

From this point on, we will work with different projections $\bX(t)$, $U(t)$, $Y_i(t)$,  etc., of the process $\bar \bX(t)$, without using explicit notation for the latter. It should be clear that, for example: $(\bX(t),U(t))$ is a projection of $\bar \bX(t)$; $(\bX(\infty),U(\infty))$ is the random value of $(\bX(t),U(t))$ in steady-state, with the joint distribution being a projection of the distribution of $\bar \bX(\infty)$; etc.

We now define the asymptotic regime. Consider a sequence  $r\to\infty$ of positive scaling parameters.
Customer arrival rates scale linearly with $r$: $\Lambda_i = \lambda_i r$, where $\lambda_i$ are fixed positive parameters.
For the process with a given value of $r$, all variables/quantities will have the superscript $r$.
Specifically, at time $t$: $\bX^r(t)$ is the state of occupied servers (without regard of their ranking),
$U^r(t)$ the location of the right-most occupied server, 
$Y^r_i(t) \equiv \sum_{\bk\in\ck} k_i X^r_{\bk}(t)$ the total number
of customers of type $i$, $Z^r(t) \equiv \sum_i Y^r_i(t)$ is the total number
of all customers, $Q^r(t) \equiv \sum_{\bk\in\ck} X_{\bk}^r(t)$ is the total number of occupied servers.
(Note that $X_{\bZero}^r(t)$ is {\em not} a component of $\bX^r(t)$.)
The steady-state value of a random element, for example, $(\bX^r(t),U^r(t))$ is denoted $(\bX^r(\infty),U^r(\infty))$.

Since arriving customers are placed for service immediately
and their service times are independent of each other and of the rest of the system, 
 $Y^r_i(\infty)$ is a Poisson random variable with mean $\rho_i r$, where $\rho_i\doteq \lambda_i/\mu_i$.
Moreover, $Y^r_i(\infty)$ are independent across $i$.
We have a trivial upper bound,
$Q^r(\infty)  \le Z^r(\infty) = \sum_i Y^r_i(\infty)$ on the total number of occupied servers, 
where $Z^r(\infty)$ has Poisson distribution with mean $r \sum_i \rho_i$.
From now on, WLOG, we assume $\sum_i \rho_i=1$; this is equivalent to rechoosing the parameter $r$ to be $r \sum_i \rho_i$.

We now define, for each $r$, the  {\em fluid-scaled} process. For any $t$ and $\bk \in \bar \ck$, define
$$
x_{\bk}^r(t) \doteq X^r_{\bk}(t)/r,
$$
 and denote  
 $\bx^r(t) = \{x^r_{\bk}(t), ~\bk\in \ck\}$. 
 For any $r$, $\bx^r(t)$ takes values in the non-negative orthant $\R_+^{|\ck|}$,
  Similarly, $y^r_i(t)=Y^r_i(t)/r$, $z^r(t)=Z^r(t)/r$, 
$x^r_{(i)}(t)=X^r_{(i)}(t)/r$, $q^r(t)=Q^r(t)/r$, and $u^r(t)=U^r(t)/r$.
Steady-state values are denoted $\bx^r(\infty), u^r(\infty)$, etc.
 
Since $q^r(\infty)=\sum_{\bk\in \ck}  x_{\bk}^r(\infty) \le z^r(\infty)=Z^r(\infty)/r$,
we see that the random variables
$q^r(\infty)$ are uniformly integrable in $r$.
This in particular implies that the sequence of distributions of 
$(\bx^r(\infty), u^r(\infty))$ 
is tight in the space $\R_+^{|\ck|} \times [\R_+ \cup \{\infty\}]$,
and therefore there always exists a limit $(\bx(\infty), u(\infty))$  
in distribution, so that 
$(\bx^r(\infty), u^r(\infty)) \implies (\bx(\infty), u(\infty))$ 
along a subsequence of $r$. Here, while $(\bx^r(\infty), u^r(\infty)) \in \R_+^{|\ck|} \times \R_+$, w.p.1,
and we must have $\bx(\infty) \in \R_+^{|\ck|}$, w.p.1, at this point it is not clear that 
$u(\infty) \in \R_+$, w.p.1, and therefore to claim convergence we need to use the one-point compactification 
$\R_+ \cup \{\infty\}$ of the space for $u^r(\infty)$.

The limit (random) vector $\bx(\infty)$ satisfies the following conservation laws:
\beql{eq-cons-laws}
\sum_{\bk\in\ck} k_i x_{\bk}(\infty) \equiv y_i(\infty) = \rho_i, ~~\forall i,
\end{equation}
and, in particular, 
\beql{eq-cons-laws2}
z(\infty)\equiv \sum_i y_i(\infty) \equiv \sum_i \rho_i
= 1.
\end{equation}
Therefore, the values of $\bx(\infty)$ are confined to the convex compact 
$(|\ck|-I)$-dimensional 
polyhedron
$$
\cx \equiv \{\bx\in \R_+^{|\ck|} ~|~ \sum_{\bk\in \ck} k_i x_{\bk} = \rho_i, ~\forall i\in\ci \}.
$$
We will slightly abuse notation by using symbol $\bx$ (and later $\tilde \bx$) for a generic element of $\R^{|\ck|}$; 
while $\bx(\infty)$ and $\bx(t)$ (and later $\tilde \bx(t)$), refer to random  elements taking values in $\R^{|\ck|}$. 

Also note that under GRAND($aZ$), $x^r_{\bZero}(\infty) \implies x_{\bZero}(\infty)=a z(\infty)=a$, as $r\rightarrow \infty$.
Similarly, under GRAND($Z^p$), $x^r_{\bZero}(\infty) r^{1-p} \implies 1$, as $r\rightarrow \infty$.

The asymptotic regime and the associated basic properties \eqn{eq-cons-laws}
and \eqn{eq-cons-laws2} hold {\em for any placement algorithm}. 
 Indeed, \eqn{eq-cons-laws}
and \eqn{eq-cons-laws2} only depend on the already mentioned  fact 
that all $Y_i^r(\infty)$ are mutually independent
Poisson  random variables
with means $\rho_i r$.

Consider the following problem of minimizing the number of occupied servers, 
$\min_{\bx \in \cx}
\sum_{\bk\in\ck} x_{\bk}$,
on the fluid scale; it is a
linear program.
Denote by $\cx^* \subseteq \cx$ the set of its 
optimal solutions,
and by $q^*$ its optimal value.

For a fixed $a>0$, define the following function
\beql{eq-L-def}
L^{(a)}(\bx) = [-\log a]^{-1}\sum_{\bk\in\ck} x_{\bk} \log [x_{\bk} c_{\bk} /(e a)],
\end{equation}
where $c_{\bk} \doteq \prod_i k_i !$, $0!=1$. 
The function $L^{(a)}(\bx)$ is strictly convex in $\bx\in\R_+^{|\ck|}$.
Consider the problem $\min_{\bx\in \cx} L^{(a)}(\bx)$; it is a
 convex optimization problem.
Denote by $\bx^{*,a} \in \cx$ its unique optimal solution, and by $q^{*,a} = \sum_{\bk \in \ck} x_{\bk}^{*,a}$ the corresponding total (fluid-scaled) number of servers; also denote $x^{*,a}_{\bZero}=a$.

Since $Y_i^r(\infty)$  has Poisson distribution with mean $\rho_i r$, we have that, for any $\veps>0$, 
\beql{eq-reduction-zp111-spec}
\pr\left\{\max_i \left| Y^r_i(\infty) - \rho_i r \right|
\le r^{1/2+\veps}  \right\} \to 1.  
\eeql
This yields the following high probability lower bound on $Q^r(\infty)$, and then on $U^r(\infty)$, under {\em any} 
placement algorithm: for any $\veps>0$ 
\beql{eq-lower-bound-uni}
0 \wedge (Q^r(\infty)-r q^{*})/r^{1/2+\veps} \Rightarrow 0, ~~\mbox{and then} ~~0 \wedge (U^r(\infty)-r q^{*})/r^{1/2+\veps} \Rightarrow 0.
\eeql
Later we will need a stronger form of \eqn{eq-reduction-zp111-spec}, obtained in lemma 5 in \cite{StZh2015}: 
when the process is in steady-state,
 the following condition holds for any $\veps >0$ and $\nu>0$: 
\beql{eq-reduction-zp111}
\pr\left\{\max_i \left| Y^r_i(t) - \rho_i r \right|
\le r^{1/2+\veps}, ~~\forall t \in [0,r^\nu] \right\} \to 1. 
\eeql

\begin{prop}[From theorems 3 and 4 in \cite{StZh2013}]
\label{th-grand-fluid-convergence} 
(i) For a fixed $a>0$, consider a sequence of systems under the GRAND($a Z$) algorithm, indexed by
$r\to\infty$. Then, $\bx^r(\infty) \Rightarrow \bx^{*,a}$; in particular, $q^r(\infty) \Rightarrow q^{*,a}$.

(ii) As $a\downarrow 0$, $\bx^{*,a} \to \cx^*$;  in particular, $q^{*,a} \to q^{*}$.
\end{prop}

Our main result for GRAND($aZ$)-FF algorithm is

\begin{thm}
\label{th-grand-fluid}
For a fixed $a>0$, consider a sequence of systems under the GRAND($a Z$)-FF algorithm, indexed by
$r\to\infty$. Then, for any $\veps>0$,  
$$
|U^r(\infty)- q^{*,a} r|/r^{1/2+\veps} \Rightarrow 0. 
$$
\end{thm}

\begin{prop}[From theorem 1 in \cite{StZh2015}]
\label{th-grand-zp-convergence} 
Let parameter $p<1$ is such that $1-\kappa (1-p) > 7/8$, where $\kappa \doteq 1+\max_{\bk} \sum_i k_i$, 
or, equivalently, $p \in (1-1/(8\kappa),1)$. 
Consider a sequence of systems under the GRAND($Z^p$) algorithm, indexed by
$r\to\infty$. Then, $q^r(\infty) \Rightarrow q^{*}$.
\end{prop}

Our main result for GRAND($Z^p$)-FF is

\begin{thm}
\label{th-grand-zp-ff}
Let parameter $p \in (1-1/(8\kappa),1)$ (as in Proposition~\ref{th-grand-zp-convergence}).
Consider a sequence of systems under the GRAND($Z^p$)-FF algorithm, indexed by
$r\to\infty$. 
Then, 
in addition to the universal lower bound \eqn{eq-lower-bound-uni}, the following upper bound holds:
$$
0 \vee (U^r(\infty)-r q^{*})/r \Rightarrow 0,
$$
and then
$$
U^r(\infty)/r  
\implies q^*.
$$
\end{thm}

\section{Combination of a rank-oblivious algorithm with First-Fit}
\label{sec-reduction-uni}

\begin{thm}
\label{th-reduction-uni}
Let ALG be a placement algorithm, oblivious of server ranks, under which $\bX(\cdot)$ is a Markov process.
For each $r$, consider the process $\bX^r(\cdot)$ under ALG in stationary regime. 
Suppose, there exist positive upper bounded function $\alpha(r)$,
positive function $\beta(r)\to \infty$, positive function $\xi(r) \le C_g r$, and a subset-valued function $E(r) \subset \R^{|\ck|}$,
 such that the following conditions hold:
\beql{eq-reduction-uni000}
\lim_{r\to\infty} \log r - \alpha(r) \beta(r) = -\infty;
\eeql
\beql{eq-reduction-uni111}
\lim_{r\to\infty} \pr\left\{\bX^r(t)  \in E(r), ~~\forall t \in [0,\beta(r)]  \right\} = 1; 
\eeql
\beql{eq-reduction-uni}
\lim_{r\to\infty} \pr\{Q^r(t) \le \hat q r + \xi(r)\},  ~~\forall t \in [0,\beta(r)] \} = 1;
\eeql
\beql{eq-reduction-uni222}
\mbox{in interval $[0,\beta(r)]$, as long as $\bX^r(t) \in E(r)$, any occupied server empties at the rate at least $\alpha(r)$.}
\eeql
(More precisely, condition \eqn{eq-reduction-uni222} means the following: there exists a constant $\tau>0$ such that 
for any $t\in [0,\beta(r)-\tau]$, if a server is occupied at time $t$, then, with probability at least $\alpha(r)\tau$,
at some time $t' \in [t,t+\tau]$ either the server empties or $\bX^r(t') \not\in E(r)$.)
Then, under the algorithm ALG-FF (which is ALG combined with the FF rule for taking empty servers)
$\pr\{U^r(\infty) < \hat q r + 2\xi(r) \} \to 1$.
\end{thm}

This results formalizes the following simple argument. If the process is such that in a time interval $[0,\beta(r)]$ the number of occupied servers is at most $N' < N$, then during this time no new empty server located to the right of $N$ will be ``taken.''. If initially the number of occupied servers is $O(r)$, and any occupied server empties at rate at least $\alpha(r)$, then the expected number of occupied servers to the right of $N$ at time 
$\beta(r)$ is upper bounded by $O(r) e^{-\alpha(r)\beta(r)}$. If this upper bound vanishes as $r\to\infty$, which is equivalent to 
$\log r - \alpha(r) \beta(r) = -\infty$, then the probability of an occupied server present to the right of $N$ vanishes. 
This type of argument was used, for example, in \cite[proof of proposition 5.9]{Aldous} for the M/M/$\infty$ with ranked servers, but there 
the ``emptying rate'' $\alpha(r)$ is automatically constant, 
because any customer service completion empties its server.
In our case a server empties only when its randomly changing configuration (due to customer departures and/or arrivals) hits ``empty'' configuration $\bZero$. Moreover, the arrival rates experienced by an individual server depend on the state of the entire system. 
That is why for the purposes of proving Theorems~\ref{th-grand-fluid} and \ref{th-grand-zp-ff} we need more general conditions, 
which allow the emptying rate $\alpha(r)$ to be decreasing with $r$, and which hold with high probability (as opposed to always).

\begin{proof}[Proof of Theorem~\ref{th-reduction-uni}]  Consider the stationary version of the process in the interval $[0,\beta(r)]$. 
Consider also a modified version of the process, which evolves the same way as the original process, except when/if the event in \eqn{eq-reduction-uni111} or \eqn{eq-reduction-uni}  
is violated for the first time, the process ``stops,'' in that all servers immediately empty and stay empty until time $\beta(r)$. 
(The modified process is, of course, non-stationary.) If the original and modified processes are coupled in the natural way (so that they coincide
until and unless the event in \eqn{eq-reduction-uni111}  or \eqn{eq-reduction-uni}  is violated), we see that,
as $r\to\infty$, the probability that the realizations of the original and modified processes coincide goes to $1$.

By condition \eqn{eq-reduction-uni222}, the modified process is such that the probability that 
an initially occupied server
does not become empty within time $t$ is at most 
$$
C_g e^{-\alpha(r) t}.
$$
The probability that the number $Q^r(0)$ of initial occupied servers is less than $(\hat q +C_g) r$ goes to $1$ as $r\to\infty$.
If $Q^r(0) \le (\hat q +C_g) r$, then the expected number of those initial occupied servers, that never emptied 
by time $\beta(r)$ is at most
$$
C_g r e^{-\alpha(r) \beta(r)},
$$
and therefore vanishes as $r\to\infty$. We conclude that for the modified process, and then for the original stationary process as well, 
\beql{eq-init-empty-uni}
\pr\{ \mbox{All initially occupied servers will empty at least once in $[0,\beta(r)]$} \} \to 1, ~~r\to\infty.
\eeql
But, we also have \eqn{eq-reduction-uni}. If the event in  \eqn{eq-reduction-uni} holds, then, by the definition
of ALG-FF, no new empty server with rank $\hat q r + 2\xi(r)$ or larger will be ``taken,'' because there 
will be empty servers with ranks less than $\hat q r + 2\xi(r)$. 
But then, if the events in \eqn{eq-init-empty-uni} and \eqn{eq-reduction-uni} both hold, 
there are no occupied servers with ranks $\hat q r + 2\xi(r)$ or larger at time $\beta(r)$.
We conclude that $\pr\{U^r(\beta(r)) < \hat q r + 2\xi(r) \} \to 1$, and then $\pr\{U^r(\infty) < \hat q r + 2\xi(r) \} \to 1$.
\end{proof}

\section{Proof of Theorem~\ref{th-grand-fluid}}
\label{sec-fsp-dynamics}

\subsection{Initial steps and general proof structure}

Fix any $\veps \in (0,1/2)$ and any $\delta>0$.
The lower bound
\beql{eq-lower-proof-az}
\pr\{(U^r(\infty)-r q^{*,a})/r^{1/2+\veps} \ge - \delta\} \to 1
\eeql
follows from property \eqn{eq-reduction-az}, which we will prove later. 
(Note that \eqn{eq-lower-proof-az} does {\em not} follow from the universal lower bound \eqn{eq-lower-bound-uni},
because $q^{*,a}$ is in general greater than $q^*$.)

Consider the upper bound
\beql{eq-upper-proof-az}
\pr\{(U^r(\infty)-r q^{*,a})/r^{1/2+\veps} \le \delta\} \to 1.
\eeql
To prove \eqn{eq-upper-proof-az} we will apply Theorem~\ref{th-reduction-uni}, namely we will verify its conditions with the following choices:
 we choose $\beta(r) = r^\nu$ with $\nu \in (0,\veps)$; $\alpha(r)$ will be a positive constant specified later, 
 so that \eqn{eq-reduction-uni000} will be true; 
$\hat q = q^{*,a}$; $\xi(r) = C_g r^{1/2+\veps}$; condition \eqn{eq-reduction-uni111} will have form
\beql{eq-reduction-az}
\pr \{ |\bX^r(t) - r \bx^{*,a}| \le r^{1/2+\veps},~~\forall t \in [0,\beta(r)]\} \to 1.
\eeql
Given these choices, notice that condition \eqn{eq-reduction-uni} is implied by \eqn{eq-reduction-az}.
So, of these two conditions only \eqn{eq-reduction-az} will need to be proved. 

Let us prove that \eqn{eq-reduction-uni222} holds for some constant $\alpha(r)$.
Indeed, as long as
$|\bX^r(t) - r \bx^{*,a}| \le r^{1/2+\veps}$ holds, we have $x^r_{(i)}(t) \ge ar /2$, and then,
for any occupied server
the instantaneous rate at which a new arrival into this server
occurs is upper bounded by $\overline \lambda \doteq [\sum_i \lambda_i]r / [ar/2] = [\sum_i \lambda_i] / [(a/2)]$. The instantaneous rate 
of a customer departure from an occupied server is lower bounded by $\underline \mu \doteq \min_i \mu_i$.
We see that,
for any occupied server with $m \le \kappa$ customers, at any time, the probability that the next $m$ arrival/departure events will be all departures is at least $[\underline \mu/ (\underline \mu + \overline \lambda)]^\kappa$.
Further, if we fix any $\tau>0$, then for any occupied server with $m \le \kappa$ customers, at any time,
the probability that there will be at least $m$ arrival/departure events within time $\tau$ 
is lower bounded by the probability of a Poisson random variable with mean $\underline \mu \tau$ being at least $\kappa$.
We can conclude that, if we fix any $\tau>0$,
then for any server at any time, the probability that the server will become empty within time $\tau$
(or condition $|\bX^r(t) - r \bx^{*,a}| \le r^{1/2+\veps}$ will ``break'')
is at least some constant $C>0$. It remains to set $\alpha(r) = C /\tau$ to complete the proof of \eqn{eq-reduction-uni222}.

Thus, the proof of the upper bound \eqn{eq-upper-proof-az} reduces to verifying condition \eqn{eq-reduction-az}.
But, condition \eqn{eq-reduction-az}, specialized to just a single time point $t=0$, also implies the 
lower bound \eqn{eq-lower-proof-az}. Therefore, the proof of the entire Theorem~\ref{th-grand-fluid} reduces to verifying condition \eqn{eq-reduction-az} for a fixed $\veps \in (0,1/2)$ and function $\beta(r)=r^\nu$ with $\nu \in (0,\veps)$.

The proof of condition \eqn{eq-reduction-az} will follow the general approach of proof of theorem 10(ii) in \cite{SY2012}, for a different model. 
It involves dividing an $O(r^\nu)$-long time interval into $O(1)$-long subintervals, and then considering local fluid limits on the subintervals.
However, the local fluid limits (which will be defined shortly) for our model are completely different, and their properties need to be derived ``from scratch.'' In the rest of this section we formally prove \eqn{eq-reduction-az}.

\subsection{Local fluid limits and related properties}

To introduce local fluid limits
we will need functional strong law of large numbers-type properties, 
which can be obtained 
from the following strong approximation of Poisson processes (see, e.g. \cite[Chapters 1 and 2]{Csorgo_Horvath}):
\begin{prop}\label{thm:strong approximation-111-clean}
A unit rate Poisson process $\Pi(\cdot)$ and 
a standard Brownian motion $W(\cdot)$ can be constructed on a common
probability space in such a way that the following holds.
For some fixed positive constants $C_1$, $C_2$, $C_3$,   such that 
$\forall T>1$ and
$\forall u \geq 0$
\[
\pr\left(\sup_{0 \leq t \leq T} |\Pi(t) - t - W(t)| \geq C_1 \log T + u\right) \leq C_2 e^{-C_3 u}.
\]
\end{prop}

As in \cite{StZh2013}, we will use notation $(\bk,i)$ for the ``edge'' between configurations $\bk$ and $\bk-\be_i$ (if the latter exists).
If there is a type-$i$ customer arrival into a server in configuration $\bk-\be_i$ (changing it into $\bk$), we call this an arrival 
``along edge $(\bk,i)$;'' similarly, if there is a type-$i$ customer service completion in a server in configuration $\bk$ (changing it into $\bk-\be_i$),
we call this a departure ``along edge $(\bk,i)$.'' Denote by $\cm \doteq \{(\bk,i) ~|~ \bk \in \ck, \bk -\be_i \in \bar \ck\}$ the set of all edges.

For each $(\bk,i)\in \cm$, consider an independent
unit-rate Poisson process $\{\hat\Pi_{\bk i}(t), ~t\ge 0\}$, common for all $r$, 
which drives
departures along edge $(\bk,i)$. Namely, 
let $D^r_{\bk i}(t)$ denote the total 
number of  departures along the edge $(\bk,i)$ in $[0,t]$; then
\beql{eq-driving-dep555}
D^r_{\bk i}(t) = \hat \Pi_{\bk i} \left(\int_{0}^t X_{\bk}^r(s) k_i \mu_i ds\right) 
= \hat \Pi_{\bk i} \left(\int_{0}^t x_{\bk}^r(s) k_i \mu_i r ~ds\right).
\end{equation}
Similarly, for each $(\bk,i)\in \cm$, consider an independent
unit-rate Poisson process $\{\Pi_{\bk i}(t), ~t\ge 0\}$, common for all $r$, which drives
arrivals along edge $(\bk,i)$. Namely, 
let $A^r_{\bk i}(t)$ denote the total 
number of  arrivals along the edge $(\bk,i)$ in $[0,t]$; then
\beql{eq-driving-arr555}
A^r_{\bk i}(t) = \Pi_{\bk i} \left(\int_{0}^t \frac{X_{\bk-\be_i}^r(s)}{X_{(i)}^r(s)} \lambda_i r ~ds\right)
= \Pi_{\bk i} \left(\int_{0}^t \frac{x_{\bk-\be_i}^r(s)}{x_{(i)}^r(s)} \lambda_i r ~ds\right),
\end{equation}
with the convention for the case $X_{(i)}^r(s)=0$ such that $X_{\bk-\be_i}^r(s)/X_{(i)}^r(s)=1$ if $\bk = \be_i$,
and $X_{\bk-\be_i}^r(s)/X_{(i)}^r(s)=0$ otherwise.

Recall that $0 < \nu < \veps$. 
From Proposition~\ref{thm:strong approximation-111-clean}, by replacing $T$ with $r^{1+\nu}$ and $u$ with $r^{1/4}$,
we obtain the following 

\begin{prop}
\label{thm:strong approximation-111}
Any subsequence of $r\to\infty$, contains a further 
subsequence (with $r$ increasing sufficiently fast),
such that, w.p.1:
\beql{eq-in-prop8}
\sup_{0 \leq t \leq r^{1+\nu}} r^{-1/2-\veps/2}| \Pi_{\bk i}(  t) - t| 
\to 0, ~~~\forall (\bk,i)\in \cm,
\eeql
and analogously for $\hat \Pi_{\bk i}( \cdot)$.
\end{prop}

Proposition ~\ref{thm:strong approximation-111} says that, w.p.1, for all large $r$, uniformly in $t \in [0, r^{1+\nu}]$,
$\Pi_{\bk i}(  t)= t + o(r^{1/2+\veps/2})$. 

Denote by $\tilde \cx$ the subspace parallel to set (manifold) $\cx$, i.e.
$$
\tilde \cx \equiv \{\tilde \bx\in \R_+^{|\ck|} ~|~ \sum_{\bk\in \ck} k_i \tilde x_{\bk} = 0, ~\forall i\in\ci \}.
$$

\begin{lem}
\label{lem-lfl}
Suppose a function $h(r)$ is fixed such that $h(r) \ge r^{1/2+\veps}$ and
$h(r) = o(r)$. Consider a sequence of initial states $X^r(0)$ such that
$$
\frac{1}{h(r)} [\bX^r(0) - r \bx^{*,a}] \to \tilde \bx(0) \in 
\R^{|\ck|}.
$$
Suppose, $T>0$ and $C_1>\max_i \mu_i k_i x_{\bk}^{*,a} ~\vee~ \max_i \frac{x_{\bk-\be_i}^{*,a}}{x_{(i)}^*} \lambda_i$ 
are fixed, where $x_{(i)}^{*,a} \doteq \sum_{\bk \in \ck} k_{i}  x_{\bk}^{*,a} + a \sum_{i'} \sum_{\bk \in \ck} k_{i'} x_{\bk}^{*,a}$.
Suppose, a deterministic sequence of the driving processes' realizations  
satisfies conditions as in Proposition ~\ref{thm:strong approximation-111}, over the time interval of length $[0,C_1 T r]$, namely  
\beql{eq-fslln-lfl}
\sup_{0 \leq t \leq C_1 T r} r^{-1/2-\veps/2}| \Pi_{\bk i}( t) - t| 
\to 0,
\eeql
and similarly for $\hat \Pi_{\bk i}(t)$.
Then any subsequence of $r$ has further subsequence along which
$$
\tilde \bx^r(t) \doteq \frac{1}{h(r)} [\bX^r(t) - r \bx^{*,a}]
$$
converges to a deterministic trajectory $\tilde \bx(\cdot)$, called {\em local fluid limit (LFL)}:
\beql{eq-conv-to-lfl}
\sup_{t \in [0,T]} \| \tilde \bx^r(t) - \tilde \bx(t) \| \to 0.
\eeql
A local fluid limit $\tilde \bx(\cdot)$ has the following properties. 
It is a Lipschitz continuous trajectory in $\R^{|\ck|}$,
satisfying linear ODE
$$
\frac{d}{dt} \tilde x_{\bk}(t) = 
\sum_{i: \bk -\be_i \in \bar \ck}  \left[ \lambda_i \frac{1}{x^{*,a}_{(i)}} \tilde x_{\bk-\be_i}(t) 
- \lambda_i \frac{x^{*,a}_{\bk-\be_i}}{(x^{*,a}_{(i)})^2} \tilde x_{(i)}(t) 
- k_i \mu_i \tilde x_{\bk}(t)  \right] 
$$
\beql{eq-lfl-ode}
- \sum_{i: \bk +\be_i \in \bar \ck} 
\left[ \lambda_i \frac{1}{x^{*,a}_{(i)}} \tilde x_{\bk}(t)  
- \lambda_i \frac{x^{*,a}_{\bk}}{(x^{*,a}_{(i)})^2} \tilde x_{(i)}(t)  
- (k_i+1) \mu_i \tilde x_{\bk+\be_i}(t)   \right], ~\bk \in \ck,
\eeql
where
 $\tilde x_{\bZero}(t) \doteq a \sum_{i} \sum_{\bk \in \ck} k_i \tilde x_{\bk}(t)$, 
$\tilde x_{(i)}(t) \doteq \sum_{\bk \in \ck} k_{i} \tilde x_{\bk}(t) + \tilde x_{\bZero}(t)$.

Furthermore, if $\tilde \bx(0) \in \tilde \cx$, then the ODE \eqn{eq-lfl-ode} solution stays in subspace $\tilde \cx$ for all $t$.
Moreover, the ODE \eqn{eq-lfl-ode},  restricted to $\tilde \cx$, 
which can be written in matrix form as $(d/dt) \tilde \bx(t) = A \tilde \bx(t)$, is stable, namely the matrix $A$ is Hurwitz -- its all eigenvalues have negative real parts.
\end{lem}

\begin{proof}[Proof of Lemma~\ref{lem-lfl}]
Let $C>0$ be a large fixed number; how large, we will specify later. 
For each realization consider the stopping time  
$\theta=\theta^r=T \wedge \min \{ t\ge 0 ~|~ \|\tilde \bx^r(t)\| > C \|\tilde \bx(0)\|\}$.
Consider the trajectory $X^r_{\bk}(t)$ on $[0,T]$, ``frosen'' starting time $\theta$, that is trajectory $X^r_{\bk}(t \wedge \theta)$ on $[0,T]$.
By \eqn{eq-fslln-lfl}, we have
$$
\Pi_{\bk i}( t) = t + o(h(r)), ~~ \hat \Pi_{\bk i}( t) = t + o(h(r)),
$$
uniformly in $t \leq C_1 T r$, $\bk$ and $i$. Therefore, we can write
$$
X^r_{\bk}(t \wedge \theta) = X^r_{\bk}(0) +
$$
$$
\sum_{i: \bk -\be_i \in \bar \ck}  \int_0^{t \wedge \theta} 
\left[ \frac{X^r_{\bk-\be_i}(s)}{X^r_{(i)}(s)} \lambda_i r - k_i \mu_i X^r_{\bk}(s) \right] ds
- \sum_{i: \bk + \be_i \in \bar \ck}  \int_0^{t \wedge \theta} 
\left[ \frac{X^r_{\bk}(s)}{X^r_{(i)}(s)} \lambda_i r - (k_i+1) \mu_i X^r_{\bk+\be_i}(s) \right] ds + o(h(r)).
$$
Then, by subtracting $x^{*,a}_{\bk} r$ from both sides above,
by subtracting and adding term
$$
\frac{x^{*,a}_{\bk-\be_i}}{x^{*,a}_{(i)}} \lambda_i r = k_i \mu_i x^{*,a}_{\bk} r
$$
in the first bracket, and analogous term in the second bracket, and normalizing everything by factor $1/h(r)$,
we obtain
$$
\tilde x^r_{\bk}(t \wedge \theta) = \tilde x^r_{\bk}(0) + o(1)
$$
$$
+ \frac{1}{h(r)} 
\sum_{i: \bk -\be_i \in \bar \ck}  \int_0^{t \wedge \theta}
\left[ \lambda_i \frac{1}{x^{*,a}_{(i)}} \tilde x^r_{\bk-\be_i}(s) \frac{h(r)}{r}
- \lambda_i \frac{x^{*,a}_{\bk-\be_i}}{(x^{*,a}_{(i)})^2} \tilde x^r_{(i)}(s) \frac{h(r)}{r}
- k_i \mu_i \tilde x^r_{\bk}(s) \frac{h(r)}{r}  + o(\frac{h(r)}{r}) \right] r ~ds
$$
$$
- \frac{1}{h(r)} 
\sum_{i: \bk +\be_i \in \bar \ck}  \int_0^{t \wedge \theta} 
\left[ \lambda_i \frac{1}{x^{*,a}_{(i)}} \tilde x^r_{\bk}(s) \frac{h(r)}{r}
- \lambda_i \frac{x^{*,a}_{\bk}}{(x^{*,a}_{(i)})^2} \tilde x^r_{(i)}(s) \frac{h(r)}{r}
- (k_i+1) \mu_i \tilde x^r_{\bk+\be_i}(s) \frac{h(r)}{r}  + o(\frac{h(r)}{r}) \right] r ~ds .
$$
Moving the factor $1/h(r)$ into the integrands, and using the definition of $\theta$, we see that the integrands are uniformly bounded in $[0,t \wedge \theta)$. Therefore, uniformly in $r$, all trajectories are Lipschitz, and we can choose a subsequence of $r$,
such that 
$$
\sup_{t \in [0,T]} \| \tilde \bx^r(t\wedge \theta) - \tilde \bx(t) \| \to 0,
$$
where $\tilde \bx(\cdot)$ is Lipschitz. Moreover, the limit trajectory $\tilde \bx(\cdot)$ must satisfy the linear ODE \eqn{eq-lfl-ode}, with 
 $\tilde x_{0}(t) \doteq a \sum_{i} \sum_{\bk \in \ck} k_i \tilde x_{\bk}(t)$ and
$\tilde x_{(i)}(t) \doteq \sum_{\bk \in \ck} k_{i} \tilde x_{\bk}(t) + a \sum_{i'} \sum_{\bk \in \ck} k_{i'} \tilde x_{\bk}(t)$,
 up to the time when possibly $\|\tilde \bx(t)\|$ hits level $C \|\tilde \bx(0)\|$. 
 But, we can always choose constant $C$ large enough, so that $\|\tilde \bx(t)\|$ does not hit level $C \|\tilde \bx(0)\|$ in $[0,T]$ (because ODE is linear). Thus, we obtain the convergence \eqn{eq-conv-to-lfl}. 
 
 Observe that ODE \eqn{eq-lfl-ode} is exactly the same as the one obtained as a linearization of the {\em fluid limit} $\bx(t)$ dynamics 
 (see the fluid limit results under GRAND($aZ$) in section 4 of \cite{StZh2013})
in the vicinity of the fixed point $\bx^{*,a}$. Namely, 
$$
\frac{d}{dt} \tilde \bx = \lim_{\delta\downarrow 0} \frac{1}{\delta} \frac{d}{dt} \bx \vert_{\bx=\bx^*+\delta \tilde \bx}, ~~\tilde \bx \in \R^{|\ck|}.
$$
(The fact that the linear ODE, describing the LFL, has the form \eqn{eq-lfl-ode}, cannot be {\em derived} by the linearization of fluid limit (non-linear) dynamics around the equilibrium point, because the LFL is obtained under a scaling ``finer'' than the fluid scaling.) 
Analysis in \cite{StZh2013} shows that fluid limits $\bx(\cdot)$ are attracted to $\cx$ exponentially: for each $i$, if $y_i(t) \doteq \sum_{\bk} k_i x_{\bk}(t)$
denotes the total ``amount" of class $i$ customers at time $t$, then 
$(d/dt)[y_i(t) -\rho_i] = -\mu_i [y_i(t) -\rho_i]$. This implies that (and also can be verified directly) that ODE \eqn{eq-lfl-ode} is such that
$\tilde y_i(t) \doteq \sum_{\bk} k_i \tilde x_{\bk}(t)$ is attracted to $\tilde \cx$ exponentially: for each $i$, 
$(d/dt) \tilde y_i(t) = -\mu_i y_i(t)$. Therefore, if $\tilde \bx(0) \in \tilde \cx$, the ODE \eqn{eq-lfl-ode} solution stays within $\tilde \cx$.

It remains to show that 
\beql{eq-lfl-ode555}
\mbox{the ODE \eqn{eq-lfl-ode} within $\tilde \cx$,  $(d/dt) \tilde \bx = A \tilde \bx$, is such that matrix $A$ is Hurwitz.}
\eeql
We know from \cite{StZh2013} (proof of lemma 7) that function $L^{(a)}(\bx)$ serves as a Lyapunov function for fluid limit $\bx(\cdot)$ trajectories. 
Namely, for any $\bx(t) \in \cx$, 
$$
\frac{d}{dt} L^{(a)}(\bx(t)) = \nabla L^{(a)}(\bx(t)) \cdot \bx'(t) < 0,
$$
unless $\bx(t)=\bx^{*,a}$. Note that function $L^{(a)}(\bx)$ is twice continuously differentiable (and in fact infinitely continuously differentiable). Therefore, on $\cx$ in the vicinity of $\bx^{*,a}$, function $L^{(a)}(\bx)$ has the following quadratic approximation:
$$
\Psi(\tilde \bx) \doteq \frac{1}{2} \tilde \bx^T \nabla^2 L^{(a)}(\bx^{*,a}) \tilde \bx = \lim_{\delta\downarrow 0} \frac{1}{\delta^2} L^{(a)}(\bx^{*,a}+ \delta \tilde \bx),  ~~\tilde \bx \in \tilde \cx.
$$
To prove \eqn{eq-lfl-ode555}
it suffices to show that $\frac{d}{dt} \Psi(\tilde \bx(t))<0$ for any non-zero $\tilde \bx(t) \in \tilde \cx$. We have
$$
\frac{d}{dt} \Psi(\tilde \bx(t)) = \nabla \Psi(\tilde \bx(t)) \cdot \tilde \bx'(t)
= \nabla [\tilde \bx^T(t) \frac{1}{2} \nabla^2 L^{(a)}(\bx^{*,a}) \tilde \bx(t)] \cdot \tilde \bx'(t)
= \nabla^2 L^{(a)}(\bx^{*,a}) \tilde \bx(t) \cdot \tilde \bx'(t)=
$$
$$
\lim_{\delta\downarrow 0} \frac{1}{\delta} \nabla L^{(a)}(\bx^{*,a}+\delta \tilde \bx(t)) ~ 
\lim_{\delta\downarrow 0} \frac{1}{\delta} \frac{d}{dt} x \vert_{\bx=\bx^*+\delta \tilde \bx(t)} = 
$$
$$
\lim_{\delta\downarrow 0} \frac{1}{\delta^2} \nabla L^{(a)}(\bx^{*,a}+\delta \tilde \bx(t)) ~ 
 \frac{d}{dt} x \vert_{\bx=\bx^{*,a}+\delta \tilde \bx(t)} = 
\lim_{\delta\downarrow 0} \frac{1}{\delta^2} \frac{d}{dt} L^{(a)}(\bx) \vert_{\bx=\bx^*+\delta \tilde \bx(t)}.
$$
Using expression for $\frac{d}{dt} L^{(a)}(\bx(t))$ (see (35) in section 4 of \cite{StZh2013}), we easily 
see that the RHS of the last display is strictly negative, unless $\tilde \bx(t)=0$. 
\end{proof}

\subsection{Completion of the proof of Theorem~\ref{th-grand-fluid} -- verification of \eqn{eq-reduction-az}}

We see from Lemma~\ref{lem-lfl} that there exist constants $C>0$ and $T\ge 1$ such that any local fluid limit 
with $\tilde \bx(0) \in \tilde \cx$
is such that
\beql{eq-lfl-decreases}
\mbox{$\| \tilde \bx(t)\| \le C \| \tilde \bx(0)\|$ for all $t$, and $\| \tilde \bx(t)\| \le \| \tilde \bx(0)\|/2$ for all $t \ge T$.}
\eeql

 Fix any $\veps \in (0,1/2)$ and function $\beta(r)=r^\nu$ 
 with $\nu \in (0,\veps)$. Fix $T\ge 1$ so that \eqn{eq-lfl-decreases} holds for any LFL with $\tilde x(0) \in \tilde \cx$.

For each $r$ consider the process is stationary regime.
By Proposition~\ref{th-grand-fluid-convergence}(i),
\beql{eq-steady-state-conv}
\pr \{ \|\bX^r(0) - r \bx^{*,a}\| \le g(r)\} \to 1,
\eeql
for some function $g(r) = o(r)$, for which WLOG we assume $r^{1/2+\veps} = o(g(r))$.
From \eqn{eq-reduction-zp111}, for any $\veps>0$ and any $\nu>0$,
\beql{eq-mminf}
\pr \{\sup_{0 \le t \le 2T r^\nu} |Y^r(t) - r \rho_i | \le r^{1/2+\veps/2}\} \to 1, ~~\forall i.
\eeql

To prove \eqn{eq-reduction-az}, 
it suffices to show that from any 
subsequence of $r$ we can find a further subsequence, along which 
\beql{eq-cond777}
\|\bX^r(t) - r \bx^{*,a} \| \le r^{1/2+\veps}, \forall t \in [T r^\nu, 2T r^\nu], ~\mbox{w.p1 for all large $r$.}
\eeql
Consider any fixed subsequence.
First, we can and do choose a subsequence along which, w.p.1,
the events in \eqn{eq-steady-state-conv} and \eqn{eq-mminf}
hold for all large $r$.

Next, fix any $\nu' \in (\nu,\veps)$. By Proposition~\ref{thm:strong approximation-111},
we can and do choose a further
subsequence, with $r$ increasing sufficiently fast, so that, 
\beql{eq-strong-approx}
\mbox{w.p.1, \eqn{eq-in-prop8} holds with $\nu$ replaced by $\nu'$.}
\eeql

We consider the process in the interval $[0, 2 T r^\nu]$,
subdivided into $2 r^\nu$ subintervals, each being $T$-long.
(To be precise, we need to consider an integer number, say $\lfloor 2 r^\nu \rfloor$,
of subintervals. This does not cause any difficulties besides making notation
cumbersome.)
In each of the subintervals $[(j-1)T,jT]$, $j=1,2,\ldots, 2r^\nu$, 
we consider the process with the time origin reset to $(j-1)T$ and the
corresponding initial state $\bX^r((j-1)T)$; and if $\|\bX^r((j-1)T) - r \bx^{*,a} \| \le g(r)$,
then we set $h(r)=\|\bX^r((j-1)T) - r \bx^{*,a} \| \vee r^{1/2+\veps}$.
(If $\|\bX^r((j-1)T) - r \bx^{*,a} \| > g(r)$ we set $h(r)=g(r)$ for completeness.)
 We consider the corresponding local fluid scaled processes $\tilde \bx^r(\cdot)$,
with their corresponding $h(r)$,
on each of the subintervals.

It follows from \eqn{eq-strong-approx} that, w.p.1, for all large $r$, condition \eqn{eq-fslln-lfl} holds simultaneously
for all subintervals. We now claim that, w.p.1, for all large $r$, the following property 
holds {\em simultaneously} for all intervals $[(j-1)T,jT]$, $j=1,2,\ldots, 2 r^\nu$. {\em If
$\|\bX^r(t)- r \bx^{*,a}\| \le C g(r)$ for all $t \le (j-1)T$, then: \\
if $\|\bX^r((j-1)T) - r \bx^{*,a} \| \in [r^{1/2+\veps}, g(r)]$ then 
$\|\bX^r(jT) - r \bx^{*,a} \| \le (1/2) \|\bX^r((j-1)T) - r \bx^{*,a} \|$ and $\|\bX^r(t) - r \bx^{*,a} \| \le C g(r)$ for $t \in [(j-1)T,jT]$;\\
if $\|\bX^r((j-1)T) - r \bx^{*,a} \| < r^{1/2+\veps}$ then $\|\bX^r((j-1)T) - r \bx^{*,a} \|  \le C r^{1/2+\veps}$ for $t \in [(j-1)T,jT]$.} \\
Indeed, if not, for arbitrarily large $r$ there would exist an $j'=j'(r)$, which is the smallest $j$ for which the above property fails; then
we would be able to find a subsequence of $r$ along which there is a convergence to an LFL on interval $[0,T]$, this LFL violating the 
properties established in Lemma~\ref{lem-lfl}. The above property implies that, w.p.1 for all large $r$, 
$\|\bX^r(t) - r \bx^{*,a} \| \le C g(r)$ for all $t \le 2T r^\nu$ and  $\|\bX^r(jT) - r \bx^{*,a} \| \le r^{1/2+\veps}$ for at least one $j \le r^\nu$
(because $(1/2)^{r^\nu} < r^{1/2+\veps}/r$).
We conclude that (w.p.1 for all large $r$) condition $\|\bX^r(jT) - r \bx^{*,a} \| \le C r^{1/2+\veps}$ 
holds for all $j \ge r^\nu$. Finally, this implies that (w.p.1 for all large $r$) condition $\|\bX^r(t) - r \bx^{*,a} \| \le C^2 r^{1/2+\veps}$ 
holds for all $t \in [T r^\nu, 2T r^\nu]$, which proves \eqn{eq-cond777}.
(Factor $C^2$ does not matter, because $\veps$ can be arbitrarily small.)
This completes the proof of \eqn{eq-reduction-az}, and of Theorem~\ref{th-grand-fluid}.
$\Box$

\section{Proof of Theorem~\ref{th-grand-zp-ff}}
\label{sec-zp-ff-proof}

\subsection{Initial steps and general proof structure}

Fix any $\delta>0$. We need to prove that
\beql{eq-upper-proof-zp}
\pr\{(U^r(\infty)-r q^{*})/r \le \delta \} \to 1.
\eeql
To prove \eqn{eq-upper-proof-zp} we will apply Theorem~\ref{th-reduction-uni}. Namely, we will 
show the existence of parameters $\nu>0$ and $\veps>0$ such that conditions of Theorem~\ref{th-reduction-uni}
hold with the following choices:
 $\beta(r) = r^\nu$; $\alpha(r)=C_g r^{\kappa(p-1)}$, 
 so that \eqn{eq-reduction-uni000} will be true; 
$\hat q = q^{*}$; $\xi(r) = (\delta/2) r$; as condition \eqn{eq-reduction-uni111} we will have \eqn{eq-reduction-zp111} - it holds for any $\veps>0$.

Let us prove condition \eqn{eq-reduction-uni222}. As long as
$\max_i \left| y_i^r(t) - \rho_i \right|
\le r^{-1/2+\veps}$ holds, we have $x^r_{(i)}(t) \ge C_g r^{p-1}$, and then,
for any occupied server
the instantaneous rate at which a new arrival into this server
occurs is upper bounded by $[\sum_i \lambda_i]r / [C_g r^p]$. The instantaneous rate 
of a customer departure from an occupied server is lower bounded by $\min_i \mu_i$. We can conclude that,
for any occupied server with $m \le \kappa$ customers, at any time, the probability that the next $m$ arrival/departure events will be all departures
is at least $C_g r^{\kappa(p-1)}$. 
From here we see that, if we fix any $\tau>0$,
then for any server at any time, the probability that the server will become empty within time $\tau$
(or condition 
$\max_i \left| y^r_i(t) - \rho_i \right|
\le r^{-1/2+\veps}$ will ``break'')
is at least $C r^{\kappa(p-1)}$. It remains to set $\alpha(r) = C r^{\kappa(p-1)} /\tau$ to complete the proof of \eqn{eq-reduction-uni222}. 

Therefore, the proof of Theorem~\ref{th-grand-zp-ff} reduces to verifying condition \eqn{eq-reduction-uni}, 
which takes form
\beql{eq-reduction-zp}
\pr\{ q^r(t) \le q^*+\delta, ~~\forall t \in [0,r^\nu]\} \to 1,
\eeql
under appropriate choice of parameters $\nu>0$ and $\veps>0$, to be specified later in the proof.

\subsection{Completion of the proof of Theorem~\ref{th-grand-zp-ff} -- verification of \eqn{eq-reduction-zp}}

Condition \eqn{eq-reduction-zp} is proved by slightly extending the proof of theorem 1 in section 5 
in \cite{StZh2015}. Initial part of this proof repeats verbatim sections 5.1-5.3 in \cite{StZh2015},
so we do not reproduce it here -- the reader is referred to \cite{StZh2015}. 
(That part: specifies the Lyapunov function used in the proof, namely $L(\bx) = L^{(r^{p-1})}(\bx)$,
and gives its relevant properties; gives certain conditions that hold in steady-state with high probability;
defines an ``artificial'' version of the process, for which the above conditions are ``enforced'' at all times;
derives the Lyapunov function drift estimates for the artificial process. 
The realizations of the artificial and original process coincide
as long as the latter happens to satisfy the above conditions.) 
The remaining part of the proof, which is a modification of section 5.4 in \cite{StZh2015},
we give now in detail, to avoid any confusion. 
We note that in \cite{StZh2015}
$L^*$ is what in this paper is denoted by $q^*$; so, notations $L^*$ and $q^*$ are used interchangeably.
Also, as in the proof in sections 5 in \cite{StZh2015}, superscript $r$ in a process notation is dropped to simplify notation;
in particular, we write $\sum_{\bk \in \ck} x_{\bk}(t) \equiv q(t)$ instead of $\sum_{\bk \in \ck} x_{\bk}^r(t) \equiv q^r(t)$.

Let the constants 
$\veps>0$, $\eta > 0$ and $c>0$ be those chosen  in section 5.3 
in \cite{StZh2015}. 
More specifically,  
$\veps > 0$ satisfies conditions (59) and (60) in \cite{StZh2015},  
$\eta \in (0, 1/4)$, 
and $c>0$ is such that lemma 6  
in \cite{StZh2015} holds. We choose $\nu=2(1-s)+\veps$.

Fix $C' >0$. (The exact choice of $C'$ will be specified later.) 
Consider the stationary version of the original process on the interval $[0,2 C' r^{-2s+2+\veps}] = [0,2 C' r^{\nu}]$,
subdivided into $2 C' r^{-3s+3+\veps}$ consecutive $r^{s-1}$-long intervals, which will be called {\em subintervals}.
The proof of \eqn{eq-reduction-zp} will be completed if we prove the following 

{\em Assertion.} For any fixed $\gamma>0$ and any fixed subsequence of $r$, there exists a further subsequence of $r$,
along which, w.p.1, for all sufficiently large $r$ the following properties hold for the original process:\\
(a) properties (41)-(43) in \cite{StZh2015}  
hold at all times within all $ 2 C' r^{-3s+3+\veps}$ subintervals, i.e. for all $t \in [0,2C' r^\nu]$;\\
(b) property $\sum_{\bk \in \ck} x_{\bk}(t) - L^* < \gamma$ holds at all times within all subintervals, starting subinterval $C' r^{-3s+3+\veps}+1$, 
i.e. for all $t \in [C' r^\nu,2C' r^\nu]$.

To prove the Assertion, fix any $\gamma>0$ and any subsequence of $r$. 
In addition to the original process, consider the artificial process, coupled to it as follows.
The initial state of the artificial process is equal (w.p.1) to that of the original one.
If the initial state (of both processes) satisfies (41)-(43) in \cite{StZh2015}, 
then
the artificial process evolves as it is defined, and it is coupled to be equal to the original process until the first time 
when (41)-(43) in \cite{StZh2015}   
is violated (for the original process).
By convention, if the initial state (of both processes) violates (41)-(43) in \cite{StZh2015},  
then
the artificial process is ``frozen'', i.e. remains equal to the initial state at all times.

Let us focus on the artificial process, and
apply the estimate (61) in \cite{StZh2015}  
to each subinterval. 
(When (61) in \cite{StZh2015} is applied to a given subinterval, the time is shifted so that $t=0$ is the beginning of that subinterval.) Then, by (61) in \cite{StZh2015} and a simple union bound, 
we see that the probability that the event in the brackets in the LHS of (61) in \cite{StZh2015} holds for at least one
of the subintervals is upper bounded by
$$
C_g  2 C' r^{-3s+3+\veps} r^{4-5s+2\veps}/\eta^2 = C_g  2 C' r^{7-8s+3\veps}/\eta^2.
$$
Since $\veps$ satisfies (60) in \cite{StZh2015},  
$C_g  2 C' r^{7-8s+3\veps}/\eta^2 \to 0$ as $r \to \infty$. 
Consider a further subsequence of $r$, increasing fast enough, e.g. $r=r(n) \ge e^n$,
so that the sum of these probabilities is finite.
 Then, for the artificial process,
by Borel-Cantelli lemma, 
w.p.1, for all large $r$, the condition
\begin{equation}
\label{eq-F-AL-diff}
\max_{0\le t \le T} \left|F(t) - \int_0^t \Xi(\bx(t)) dt\right| < 2 \eta r^{3s-3-\veps}
\end{equation}
holds simultaneously for all $2 C' r^{-3s+3+\veps}$ subintervals; 
furthermore, we have (62)-(63) in \cite{StZh2015}  
for all these subintervals simultaneously.

By lemmas 5 and 6 in \cite{StZh2015}  
and Borel-Cantelli lemma, we can choose a further subsequence of $r$, along which
the original process is such that w.p.1, for all large $r$, conditions (41)-(43) in \cite{StZh2015}  
hold on all $2C' r^{-3s+3+\veps}$ subintervals, and therefore the 
artificial process and the original process coincide. This proves Assertion (a).
Furthermore, along the last chosen subsequence, 
w.p.1, for all large $r$,
$F(t)$ (which we defined for the artificial process) is
equal to the increment of $L$, $F(t) = L(\bx(t)) -L(\bx(0))$, for the original process,
for all subintervals simultaneously, and we also have
(62)-(63) in \cite{StZh2015}  
for all subintervals simultaneously. 

Then, w.p.1, for all large $r$, the following occurs for the original process.
If at the beginning of a subinterval, $\Delta L \equiv L - L^* \ge \gamma$,
then either condition $\Delta L \le \gamma$ is ``hit'' within the subinterval, or
at the end of the subinterval $\Delta L$ is smaller by at least 
$r^{3s-3-\veps}/2 > 0$.
This follows from (63) in \cite{StZh2015}, \eqn{eq-F-AL-diff},   
condition $\eta < 1/4$,
and the fact that (by Lemma 4 in \cite{StZh2015}) 
$\Delta L \ge \gamma$ implies that
$|\chi_{\bk,\bk',i}| \ge \delta_1 >0$
for some $\delta_1=\delta_1(\gamma)$ that (just like $\gamma$) 
 does not depend on $r$.
In addition, 
if at the beginning of a subinterval or any other point in it $\Delta L \le \gamma$, then in this entire subinterval $L(t) - L(0) \le \gamma/2$.

Note that w.p.1, for all large $r$, at the beginning of the first subinterval,
$L(\bx(0)) - L^* \le C''$ for some constant $C''>0$ independent of $r$. (We know from \cite{StZh2015} that function $L(\bx) = L^{(r^{p-1})}(\bx)$,
and it converges to $\sum_{\bk\in\ck} x_{\bk}$ uniformly on compact sets.) 
We now specify the choice of $C'$: it is any constant satisfying $C' > \max\{2 C'',1\}$.
Given this choice, we see that (w.p.1, for all large $r$) condition $\Delta L \le \gamma$ is 
in fact ``hit'' within one of the first $C' r^{-3s+3+\veps}$ subintervals; this in turn implies that 
 $\Delta L < 2\gamma$ must hold at the end of subinterval $C' r^{-3s+3+\veps}$, i.e. at time $t=C' r^\nu$.
 Recall again (\cite{StZh2015}) that function $L(\bx) = L^{(r^{p-1})}(\bx)$ converges to $\sum_{\bk\in\ck} x_{\bk}$ uniformly on compact sets.
 We finally obtain that, w.p.1, for all large $r$, 
 $\sum_{\bk \in \ck} x_{\bk}(t) - L^* < 3\gamma$ within all subintervals, starting subinterval $C' r^{-3s+3+\veps}+1$, i.e. at all times $t \in [C' r^\nu,2C' r^\nu]$. Rechoosing $\gamma$ completes the proof of Assertion (b),
 and then of \eqn{eq-reduction-zp}. 
$\Box$

\section{Discussion}
\label{sec-discussion}

We prove that both algorithms GRAND($aZ$)-FF and GRAND($Z^p$)-FF are asymptotically optimal in that, in steady-state, $U/r \Rightarrow q^*$. For GRAND($Z^p$)-FF Theorem~\ref{th-grand-zp-ff} shows this directly. For GRAND($aZ$)-FF Theorem~\ref{th-grand-fluid}
implies the existence of a dependence of parameter $a$ on $r$, such that $U/r \Rightarrow q^*$. (It is natural to expect that 
the dependence $a=r^{p-1}$, where $p$ is same as used in GRAND($Z^p$)-FF, does that. But, it is not formally addressed in this paper.)
Therefore, under both algorithms
a bound $U- q^* r \sim o(r)$ is achieved. The questions of more precise characterization of the $o(r)$-term under GRAND($Z^p$)-FF
and GRAND($aZ$)-FF remain open, 
and may be a subject of further research.

Comparing Theorems~\ref{th-grand-zp-ff} and \ref{th-grand-fluid}, it is worth noting the following.
GRAND($Z^p$)-FF algorithm is, in essence (but not exactly), the GRAND($aZ$)-FF with parameter $a$ depending on $r$
in the specific way, $a=r^{p-1}$. The asymptotic regime of Theorem~\ref{th-grand-fluid}, with parameter $a$ kept constant
as $r$ increases, is of independent interest.
The bound $|U - q^{*,a}| \le O(r^{1/2+\veps})$ for any $\veps>0$ in Theorem~\ref{th-grand-fluid} 
is quite tight, given that in cannot be better than $O((2 r \log \log r)^{1/2})$ (by the previous results for M/M/$\infty$ with ranked servers).

It is conjectured in \cite{StZh2013} (see conjecture 10 there) that $Q/r \Rightarrow q^*$ under GRAND($0$) algorithm, which is the instance of GRAND,
such that an arriving customer is placed uniformly at random into one of the occupied servers that can still fit it, 
and into an empty server only when none of the occupied servers are available. 
The conjecture is supported by the intuition provided by the analysis of GRAND($aZ$) and GRAND($Z^p$). Specifically, 
the optimality proofs in \cite{StZh2013} and \cite{StZh2015} show that GRAND($X_{\bZero}$) algorithm dynamics
drives the system state to the vicinity of optimal state, with the time scale of this dynamics being larger when $X_{\bZero}$ is smaller.
Those proofs suggest that the dynamics under GRAND($0$) should be similar, albeit on a different -- larger -- time scale.
(The conjecture is also supported by simulation results in \cite{StZh2013}.)
Proving the GRAND($0$) asymptotic optimality conjecture remains an interesting subject for future work. 
And if/when this conjecture is proved, there is a hope that the
approach developed in this paper can be instrumental in proving that $U/r \Rightarrow q^*$ under GRAND($0$) combined with FF.

Another potential subject  of future research is analyzing a ``Greedy-FF'' algorithm, which places an arriving customer into the left-most server available to it. Asymptotic optimality of this algorithm, in the sense of $U/r \Rightarrow q^*$, for the general model in this paper,
 would be far less intuitive than that of GRAND($0$) combined with FF, because the system dynamics under Greedy-FF is substantially different. 
 It is far less intuitive that Greedy-FF will asymptotically minimize the number of
 occupied servers, $Q/r \Rightarrow q^*$, in the first place. 
 
 Finally, although it is not the focus of this paper, we now briefly discuss a system with finite number $N$ of servers, with blocking. 
 Suppose $N = q^* r + g(r)$, for some positive subcriticality margin $g(r)$. The basic question is: under a given placement/blocking 
 algorithm, does the blocking probability vanish as $r \to \infty$? An algorithm design options for this problem depend on the 
 amount of system information available. Consider some of the possible settings.
 
 \begin{itemize}
 \item Suppose {\em all system parameters are known in advance}, including arrival rates $\lambda_i r$, service rates $\mu_i$.
 In this case we can a priori find the optimal $q^*$ and a corresponding optimal $\bx^*$, that is $q^* = \sum_{\bk \in \ck} x^*_{\bk}$.
 Then, we can ``preallocate'' servers to fixed packing configurations $\bk$ in proportion to $x^*_{\bk}$, so that the total capacity 
 available to type $i$ customers is $N_i = [\sum_{\bk \in \ck} k_i x^*_{\bk}] (r + g(r))= \rho_i (r +g(r))$. An arriving customer of type $i$ is 
 blocked if and only if type $i$ customers already occupy the entire capacity $N_i$, dedicated to it. Therefore, from the point of view of type $i$ customers, the system will operate as an independent 
 $M/M/N_i$ ``Erlang-B'' system with blocking. In particular, the blocking probability will vanish as long
 as $r^{1/2} = o(g(r))$. A downside of this approach is that the system parameters are not necessarily known in advance and/or may change with time.
 
 \item Suppose {\em the arrival rates $\lambda_i r$ are not known, but the service rates $\mu_i$ are.} Suppose further that
 we have an algorithm ALG for the system with infinite number of ranked servers, for which $[U-(N-1)] \vee 0 = [U-(q^* r + g(r)-1)] \vee 0  \Rightarrow 0$. 
 (For example, as we prove in this paper, this is the case for GRAND($Z^p$)-FF with $g(r)=O(r)$.) Then we can construct the algorithm,
 let us label it ALG-BLOCK, for the system with $N$ servers, under which the blocking probability vanishes. Specifically, ALG-BLOCK
 will emulate ALG in that it will ``pretend'' that in addition to actual servers, ranked $1, \ldots, N$, there is the infinite number of ``imaginary'' servers, ranked $N+1, N+2, \ldots$. ALG-BLOCK will work exactly like ALG, except when an arriving customer is placed into an imaginary server,
 this actual customer is blocked, but its ``imaginary version'' is ``served'' by the imaginary server. Imaginary customers ``complete service'' after a random service time generated by the algorithm. (Recall that the service time distributions are known.) 
Because the probability of not having an empty actual server vanishes, so does the blocking probability.

\item Suppose {\em neither the arrival rates $\lambda_i r$ nor the service rates $\mu_i$ are known.} In this scenario, 
a natural algorithm to consider is one of the algorithms in \cite{St2015_grand-het},
 simply assigning an arriving customer to any server currently available to it, uniformly at random, and blocks the customer if none is available.
(The algorithm can be viewed as another version of GRAND.)
The results of \cite{St2015_grand-het}, mentioned in Section~\ref{sec-rel-prev-work},
 suggest that margin $g(r)=O(r)$ is sufficient for blocking probability to vanish, but there is no proof.
  
 \end{itemize}
 
  {\em Acknowledgement.} I would like to thank the associate editor and two anonymous referees for valuable comments, which helped to improve the exposition in the paper.

%\iffalse
\bibliographystyle{acmtrans-ims}
%\bibliographystyle{apt}
%\bibliographystyle{abbrv}

%%%%%%\bibliography{biblio-stolyar}
%\bibliography{bibliography}

\begin{thebibliography}{}
\ifx \url   \undefined \def \url#1{#1}   \fi

\bibitem{Aldous}
\textsc{Aldous, D.} (1986).
\newblock {Some interesting processes arising as heavy traffic limits in an
  M/M/$\infty$ storage process}.
\newblock \emph{Stochastic Processes and their Applications\/}~\textbf{22},~2,
  291–313.

\bibitem{CKS86}
\textsc{Coffman, E.}, \textsc{Kadota, T.}, \textsc{and} \textsc{Shepp, L.}
  (1985).
\newblock {A stochastic model of fragmentation in dynamic storage allocation}.
\newblock \emph{SIAM Journal of Computing\/}~\textbf{14},~2, 416–425.

\bibitem{CL89}
\textsc{Coffman, E.} \textsc{and} \textsc{Leighton, F.} (1989).
\newblock {A provably efficient algorithm for dynamic storage allocation}.
\newblock \emph{Journal of Computer and System Sciences\/}~\textbf{38},~1,
  2–35.

\bibitem{Csorgo_Horvath}
\textsc{Cs{\"{o}}rg{\H{o}}, M.} \textsc{and} \textsc{Horv{\'{a}}th, L.} (1993).
\newblock \emph{Weighted approximations in probability and statistics}.
\newblock Wiley.

\bibitem{ErSt2024}
\textsc{Ernst, P.} \textsc{and} \textsc{Stolyar, A.~L.} (2024).
\newblock Asymptotic optimality of dynamic first-fit packing on the half-axis.
\newblock arXiv:2404.03797.

\bibitem{Gulati2012}
\textsc{Gulati, A.}, \textsc{Holler, A.}, \textsc{Ji, M.},
  \textsc{Shanmuganathan, G.}, \textsc{Waldspurger, C.}, \textsc{and}
  \textsc{Zhu, X.} (2012).
\newblock Vmware distributed resource management: Design, implementation and
  lessons learned.
\newblock \emph{VMware Technical Journal\/}~\textbf{1},~1, 45--64.

\bibitem{K00}
\textsc{Knessl, C.} (2000).
\newblock {Asymptotic expansions for a stochastic model of queue storage}.
\newblock \emph{The Annals of Applied Probability\/}~\textbf{10},~2, 592–615.

\bibitem{Kosten37}
\textsc{Kosten, L.} (1937).
\newblock {Uber Sperrungswahrscheinlichkeiten bei Staffelschaltungen}.
\newblock \emph{Electra Nachrichten-Technik\/}~\emph{14}, 5–12.

\bibitem{Newell-book}
\textsc{Newell, G.~F.} (1984).
\newblock \emph{The M/M/$\infty$ Service System with Ranked Servers in Heavy
  Traffic}.
\newblock Springer.

\bibitem{Preater97}
\textsc{Preater, J.} (1997).
\newblock {A Perpetuity and the M/M/$\infty$ Ranked Server System}.
\newblock \emph{Journal of Applied Probability\/}~\textbf{34},~2, 508–513.

\bibitem{SK08}
\textsc{Sohn, E.} \textsc{and} \textsc{Knessl, C.} (2008).
\newblock {The distribution of wasted spaces in the M/M/$\infty$ queue with
  ranked servers}.
\newblock \emph{Advances in Applied Probability\/}~\textbf{40},~3, 835–855.

\bibitem{St2015_grand-het}
\textsc{Stolyar, A.~L.} (2017).
\newblock Large-scale heterogeneous service systems with general packing
  constraints.
\newblock \emph{Advances in Applied Probability\/}~\textbf{49},~1, 61--83.

\bibitem{SY2012}
\textsc{Stolyar, A.~L.} \textsc{and} \textsc{Yudovina, E.} (2012).
\newblock Tightness of invariant distributions of a large-scale flexible
  service system under a priority discipline.
\newblock \emph{Stochastic Systems\/}~\textbf{2},~2, 381--408.

\bibitem{StZh2013}
\textsc{Stolyar, A.~L.} \textsc{and} \textsc{Zhong, Y.} (2015).
\newblock Asymptotic optimality of a greedy randomized algorithm in a
  large-scale service system with general packing constraints.
\newblock \emph{Queueing Systems\/}~\textbf{79},~2, 117--143.

\bibitem{StZh2015}
\textsc{Stolyar, A.~L.} \textsc{and} \textsc{Zhong, Y.} (2021).
\newblock A service system with packing constraints: Greedy randomized
  algorithm achieving sublinear in scale optimality gap.
\newblock \emph{Stochastic Systems\/}~\textbf{11},~2, 83--111.

\end{thebibliography}
%\fi

\end{document}